\documentclass[12pt]{article}
\usepackage{amsmath,amssymb,amsthm,appendix,hyperref}
\usepackage{graphicx}

\theoremstyle{definition}
\newtheorem{thm}{Theorem}
\newtheorem{lem}[thm]{Lemma}
\newtheorem{defn}[thm]{Definition}

\title{On Radial Colorings of Annuli}
\author{Jeremy F. Alm\footnote{Corresponding author.}\\ Illinois College\\ Jacksonville, IL 62650\\ \texttt{alm.academic@gmail.com} \and Jacob Manske\\ Texas State University\\ San Marcos, TX 78666\\ \texttt{jmanske@gmail.com}}
\date{\today}

\begin{document}
\maketitle
\thispagestyle{empty}


\begin{abstract}
  We consider the chromatic numbers of unit-distance graphs of various annuli. In particular, we consider radial colorings, which are ``nice" colorings, and completely determine the radial chromatic numbers of various annuli.
\end{abstract}


  \section{Introduction}\label{intro}

The \emph{Chromatic Number of the Plane}, denoted by $\chi(\mathbb{R}^2)$, is the least integer $N$ such that the points of the plane can be colored in $N$ colors such that no two points exactly unit distance apart are the same color.  Well-known elementary arguments show that $4\leq\chi(\mathbb{R}^2)\leq 7$, but no improvement on these bounds appears to be forthcoming.

The difficulty of improving these bounds has led several authors to
consider various modifications to the problem.  For instance, one
might restrict the type of coloring.  In \cite{Falconer81}, Falconer
shows that if the color classes must be measurable, then at least
five colors are required.  (See \cite{Soifer}, Ch.~9, for a nice
exposition.) In \cite{Townsend81}, Townsend announces the result
that every ``map-coloring" of the plane requires at least six
colors, and provides a proof in \cite{Townsend2005}. In a basically
identical result couched in different language, Coulson
\cite{Coulson04} considers colorings in which the color classes are
composed of ``tiles", and shows that six colors are required.

Another approach is to consider colorings of proper subsets of the
plane.  In \cite{Axenovichetal}, Axenovich et al. consider but do
not determine the chromatic number of $\mathbb{Q}\times\mathbb{R}$.
They also determine the chromatic number of infinite strips of
different widths, and of the union of two parallel lines.  In
\cite{Kruskal08}, Kruskal considers bounded simply connected subsets
of the plane.  In particular, he shows that for a closed disk of
radius $r$,
\begin{itemize}
  \item the disk is 2-colorable if and only if $r\leq\frac{1}{2}$;
  \item the disk is 3-colorable if and only if $r\leq\frac{1}{\sqrt{3}}$;
  \item the disk is 4-colorable if $r\leq\frac{1}{\sqrt{2}}$.
\end{itemize}

In the present paper, we consider the annulus with inner radius $\frac{1}{2}-r$ and outer radius $\frac{1}{2}+r$, where  $0<r<\frac{1}{2}$, and prove some surprising results in the case that the colorings are required to be ``nice."

\section{Arbitrary colorings of the annulus.}

Let $0<r<\frac{1}{2}$, and let $A_r=\{p\in\mathbb{R}^2:\frac{1}{2}-r\leq\|p\|\leq\frac{1}{2}+r\}$.  We recall two lemmas from Kruskal's paper \cite{Kruskal08}.
\begin{lem}\label{rod}
  A \emph{rod} is a line segment of unit length.  If a region $R$ is 2-colored, and we slide a rod continuously so that its endpoints stay within the interior of $R$, then the set of points passed over by a given endpoint is monochromatic.
\end{lem}

\begin{lem}\label{trirod}
  A \emph{tri-rod} is an equilateral triangle of unit side length.  If a region $R$ is 3-colored, and we slide a tri-rod continuously so that its vertices stay within the interior of $R$, then the set of points passed over by a given vertex is monochromatic.
\end{lem}

If $r>0$, then a rod can be placed with endpoints interior to $A_r$.  By rotating the rod by $180^\circ$ about its center, we see that if $A_r$ were 2-colored, then by Lemma \ref{rod} the two endpoints would have to be colored the same color.  But the endpoints are unit distance apart, so $A_r$ is not 2-colorable.  See Figure \ref{fig:F1}.  An alternate method of proof is to embed an odd cycle.  See Figure \ref{fig:Fodd}.

\begin{figure}[htb!]
\centering
\includegraphics[width=140pt]{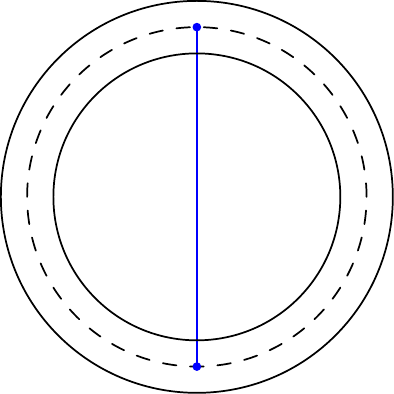}
\caption{An  embedding of a rod}
\label{fig:F1}
\end{figure}

\begin{figure}[htb!]
\centering
\includegraphics[width=140pt]{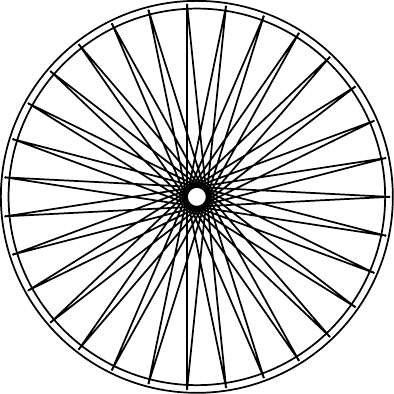}
\caption{An embedding of a odd cycle}
\label{fig:Fodd}
\end{figure}

%
%
%

If $0<r\leq\frac{2-\sqrt{3}}{2\sqrt{3}}$, then $A_r$ is 3-colorable.  Figure \ref{fig:F2} shows a proper 3-coloring of $A_r$ with $r=\frac{2-\sqrt{3}}{2\sqrt{3}}$ (so that the outer radius of $A_r$ is $\frac{1}{\sqrt{3}}$).  Each boundary between sectors is colored the same color as the sector clockwise from it.

\begin{figure}[htb!]
\centering
\includegraphics[width=140pt]{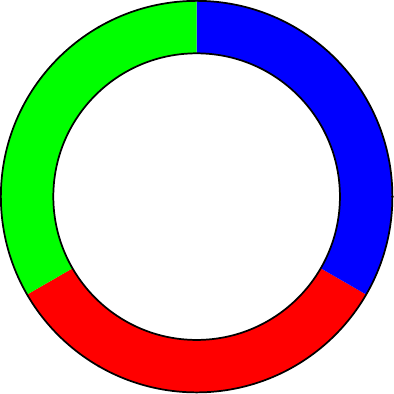}
\caption{A proper 3-coloring for $r=\frac{2-\sqrt{3}}{2\sqrt{3}}$}
\label{fig:F2}
\end{figure}

If $r>\frac{2-\sqrt{3}}{2\sqrt{3}}$, then the outer radius of $A_r$ is greater than $\frac{1}{\sqrt{3}}$, so a tri-rod can be placed such that its endpoints lie in the interior of $A_r$.  Thus by rotating $120^\circ$ and applying Lemma \ref{trirod}, we see that $A_r$ is not 3-colorable.  See Figure \ref{fig:F3}.

\begin{figure}[htb!]
\centering
\includegraphics[width=140pt]{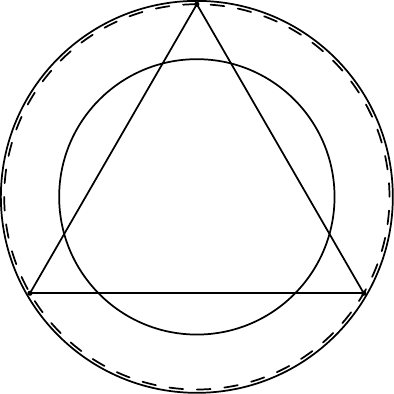}
\caption{An  embedding of a tri-rod}
\label{fig:F3}
\end{figure}

Figure \ref{fig:F4} shows a proper 4-coloring of $A_r$ for $r=\frac{2-\sqrt{2}}{2\sqrt{2}}$ (so that the outer radius of $A_r$ is $\frac{1}{\sqrt{2}}$).  Again, the boundaries between sectors are colored the same color as the adjacent sector in the clockwise direction.

\begin{figure}[htb!]
\centering
\includegraphics[width=140pt]{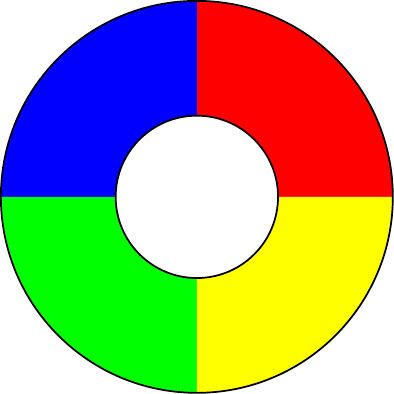}
\caption{A proper 4-coloring for $r=\frac{2-\sqrt{2}}{2\sqrt{2}}$}
\label{fig:F4}
\end{figure}

Kruskal leaves open the question of whether a closed disk of radius greater than $\frac{1}{\sqrt{2}}$ can be properly 4-colored.  One might think that an annulus with outer radius $\frac{1}{\sqrt{2}}+\varepsilon$ could be properly 4-colored---after all, it has a hole in the middle---but the present authors have been unable to construct such a coloring. The nature of the difficulty will be illuminated in the next section.

\section{Radial Colorings}

\begin{defn}
  Let $0<r<\frac{1}{2}$.  A coloring of $A_r$ is called \emph{radial} if there exists a sequence of radii $r_1,r_2,\ldots,r_n=r_1$ such that the sector strictly between $r_i$ and $r_{i+1}$ is colored with a single color.  Let $\chi^{\text{radial}}(A_r)$ denote the least number of colors needed to give a proper coloring of $A_r$ using radial colorings only.
\end{defn}
 The radial colorings are the ``nice" colorings mentioned at the end of Section \ref{intro}.  We emphasize that the boundaries between sectors can be colored arbitrarily, while the sectors between radial boundaries must be monochromatic.  See Figure \ref{fig:Fsector}.

 \begin{figure}[htb!]
\centering
\includegraphics[width=140pt]{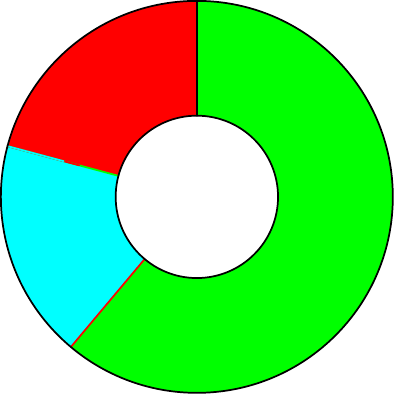}
\caption{A radial coloring}
\label{fig:Fsector}
\end{figure}

\begin{defn}
  A \emph{unit chord} in $A_r$ is a chord of unit length.  A \emph{unit sector} is a sector of minimal size that contains a unit chord.  See Figure \ref{fig:Fchord}.

\end{defn}

\begin{figure}[htb!]
\centering
\includegraphics[width=140pt]{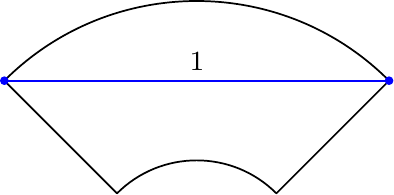}
\caption{A unit chord contained within a unit sector}
\label{fig:Fchord}
\end{figure}

The following lemma will be powerful enough to allow us to determine $\chi^{\text{radial}}(A_r)$ for all $0<r<\frac{1}{2}$.

\begin{lem}\label{mainlem}
  Let $0<r<\frac{1}{2}$ and let $A_r$ be given a proper radial coloring.  Then the interior of any color class is included in some unit sector.
\end{lem}

\begin{proof}
  Call the intersection of a radial ray with $A_r$ a \emph{radial segment}.  See Figure \ref{fig:Fradseg}.

  \begin{figure}[htb!]
\centering
\includegraphics[width=140pt]{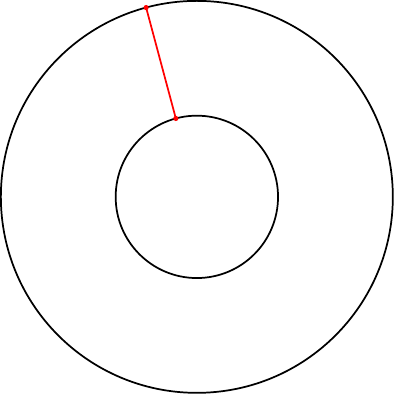}
\caption{A radial segment}
\label{fig:Fradseg}
\end{figure}

By way of contradiction, suppose there exist two radial segments of
the same  color (red, say) such that the chord joining their
farthest-apart points has length greater than 1 (Figure
\ref{fig:fgr1}).

\begin{figure}[htb!]
\centering
\includegraphics[width=140pt]{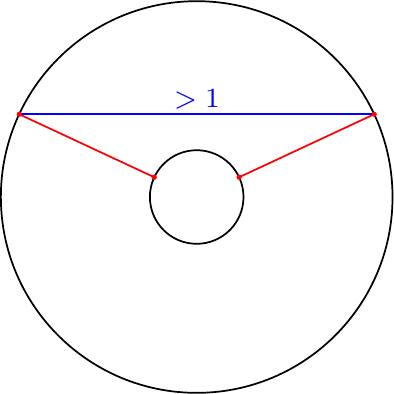}
\caption{Two radial segments}
\label{fig:fgr1}
\end{figure}

Call one of the segments $S_1$ (Figure \ref{fig:lemmafig}), and call
the point at which it intersects the outer circle $P$.  Call the
other segment $S_{2}$. Let $A$ and $B$ be the two points that are a
unit chord away from $P$ (Figure \ref{fig:AandB}).

\begin{figure}[htb!]
\centering
\includegraphics[width=140pt]{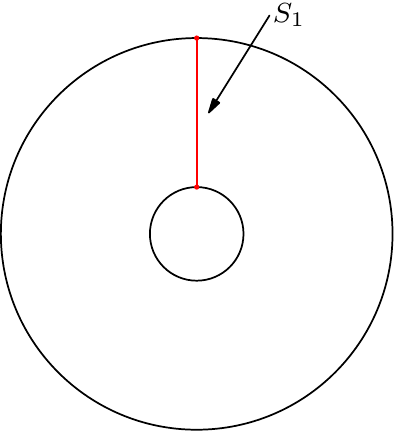}
\caption{The segment $S_{1}$}
\label{fig:lemmafig}
\end{figure}

\begin{figure}[htb!]
\centering
\includegraphics[width=140pt]{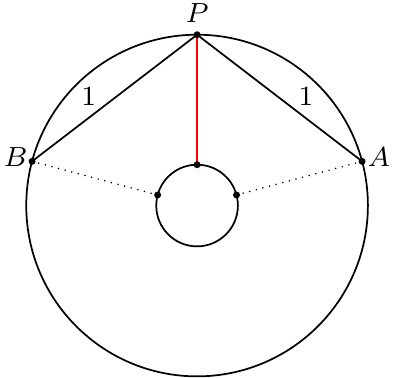}
\caption{Points $A$, $B$, and $P$}
\label{fig:AandB}
\end{figure}

 Let $P'$ be antipodal to $P$, and draw a segment from $A$ to $P'$ (Figure \ref{fig:AtoPprime}).

\begin{figure}[htb!]
\centering
\includegraphics[width=140pt]{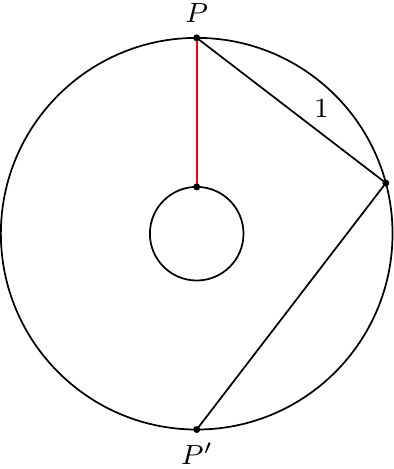}
\caption{The segment from $A$ to $P'$}
\label{fig:AtoPprime}
\end{figure}

  Now take the unit-length rod that goes from $P$ to $A$, and slide it continuously so that one endpoint stays on $S_1$ and the other stays on $AP'$ (Figure \ref{fig:segslide}).

\begin{figure}[htb!]
\centering
\includegraphics[width=140pt]{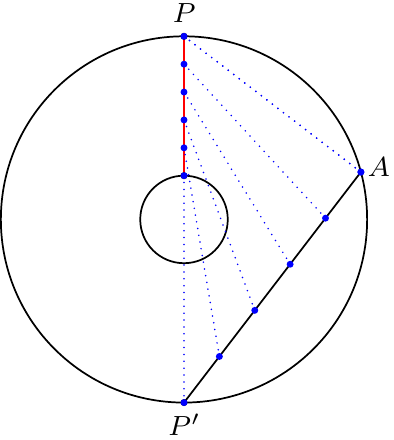}
\caption{Sliding the segment $AP$}
\label{fig:segslide}
\end{figure}

  Thus every radial segment between $A$ and $P'$ has a point on it that is unit distance from some point of $S_1$.  (Similarly for those between $P'$ and $B$.)  But this is a contradiction, since some point on $S_{2}$ must be on $AP'$ or $BP'$.
\end{proof}

\begin{thm}
  Let $0<r<\frac{1}{2}$.  Then $\chi^{\text{radial}}(A_r)$ is the least integer $N$ such that one can ``walk around" the outside circumference using $N$ unit-chord steps, i.e., $N=\lceil\frac{2\pi}{\theta} \rceil$, where $\theta=\arccos\left (1-\dfrac{1}{2(\frac{1}{2}+r)^2}\right)$.
\end{thm}



\begin{proof}
  That $\chi^{\text{radial}}(A_r)\geq N$ follows immediately from the previous lemma.  It is easy to color $A_r$ in $N$ colors using $N-1$ unit sectors as different colors and one ``leftover" sector in the $N^\text{th}$ color.
\end{proof}

Below is a table of values for $\chi^\text{radial}(A_r)$.\\
\begin{table}
\begin{center}
\begin{tabular}{|c|c|}
\hline
$0<r\leq\frac{2-\sqrt{3}}{2\sqrt{3}}$ & $\chi^\text{radial}=3$\\
\hline
$\frac{2-\sqrt{3}}{2\sqrt{3}}<r\leq \frac{2-\sqrt{2}}{2\sqrt{2}}$ & $\chi^\text{radial}=4$ \\
\hline
$\frac{2-\sqrt{2}}{2\sqrt{2}}<r\leq\frac{-1}{2}+\sqrt{\frac{2}{5-\sqrt{5}}}$ & $\chi^\text{radial}=5$\\
\hline
$\frac{-1}{2}+\sqrt{\frac{2}{5-\sqrt{5}}} <r<\frac{1}{2}$ & $\chi^\text{radial}=6$\\
\hline
\end{tabular}
\end{center}
\caption{A table of values for $\chi^\text{radial}(A_r)$}
\end{table}

We note that if we let $r=1/2 $, then we get a disk of radius 1, which can be given a proper radial  coloring by coloring the 6 unit sectors in distinct colors and the center of the disk in the 7th color.

\section{Further Directions}

It may be interesting to consider other connected but not simply
connected regions of the plane.   For instance, consider the
well-known periodic hexagonal coloring of the plane in 7 colors.
Take one color class, enlarge it slightly and delete it, leaving a
``holey plane."  Certainly, this holey plane can be colored in 6
colors.  Can it be colored in 5?  Townsend's result on colorings
using tiles would seem not to apply.

Bounded and not simply connected subsets may also be interesting to consider, although it would seem that their analysis promises little additional insight into the larger problem of understanding plane colorings and determining the chromatic number of the plane.  After all, the authors were not able to find colorings of $A_r$ that improve Kruskal's corresponding bounds for circular regions; the ``donut hole" did not seem to be useful for arbitrary colorings.

A more promising approach might be to try to prove (or find a counterexample for) a Moser-spindle version of Kruskal's tri-rod lemma.  If $r>\frac{3}{\sqrt{11}} - \frac{1}{2}$, then a spindle can be embedded into the interior of $A_r$ (see Figure \ref{fig:moserfig}).  If a tri-rod can be embedded, then $A_r$ is not 3-colorable.  Perhaps if a spindle can be embedded, $A_r$ is not 4-colorable.

\begin{figure}[htb!]
\begin{center}
\includegraphics[width=200pt]{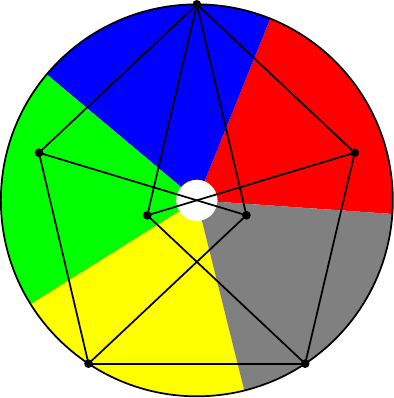}
\caption{The Moser spindle can be embedded into $A_{r}$ if $r > \dfrac{3}{\sqrt{11}} - \dfrac{1}{2}$. A radial coloring of $A_{r}$ requires 5 colors in this case.}
\label{fig:moserfig}
\end{center}
\end{figure}

\end{document}